\documentclass[11pt,xcolor=pdftex,a4paper,oneside,centertags,reqno]{amsart}

\usepackage{latexsym}\usepackage{amsmath}
\usepackage{amsfonts}
\usepackage{amssymb}
\usepackage{graphpap}
\usepackage{epsfig}
\usepackage{amscd}
\usepackage{amsthm}
\usepackage{enumerate}

\usepackage[francais]{[babel}
\usepackage[ansinew]{inputenc}
\usepackage{graphics}
\usepackage{hyperref}
\usepackage{amsfonts}
\usepackage{pdfsync}
\usepackage{color}
\usepackage{endnotes}

\newcommand{\grad}[1]{{\rm Grad}(#1)}

\newcommand{\f}[0]{\varphi}

\newcommand{\N}[0]{{\mathbb N}}

\newcommand{\pd}[0]{\partial}

\newcommand{\R}[0]{\mathbb{R}}

\newcommand{\cI}[0]{{\mathcal I}}

\newcommand{\cL}[0]{{\mathcal L}}

\newcommand{\cR}[0]{{\mathcal R}}

\newcommand{\nor}[1]{\| #1 \|}
\newcommand{\Nor}[1]{\left |\hskip -0.7pt\left | #1 \right |\hskip -0.7pt\right |}

\newcommand{\mn}[2]{\{ #1\, ;#2 \}}
\newcommand{\sk}[2]{\langle #1 , #2\rangle}
\newcommand{\Sk}[2]{\Big\langle #1 , #2\Big\rangle}

\newtheorem{theorem}{Theorem}
\newtheorem{lema}[theorem]{Lemma}
\newtheorem{proposition}[theorem]{Proposition}
\newtheorem{cor}[theorem]{Corollary}

\newtheorem{definition}[theorem]{Definition}

\renewcommand\leq[0]{\leqslant}
\renewcommand\geq[0]{\geqslant}
\renewcommand\epsilon[0]{\varepsilon}

\newcommand\wrt{\,\text{\rm d}}
\renewcommand\mod[1]{\left\vert{#1}\right\vert}
\newcommand{\esss}{\mathop{\text{ess sup}\,}}
\newcommand\norm[2]{{\left\Vert{#1}\right\Vert_{#2}}}

\newtheorem{preremark}[theorem]{Remark}  \newenvironment{remark}
{\begin{preremark}\rm}{\end{preremark}}

\begin{document}

\title[Linear dimension-free estimates in a theorem of D. Bakry]{Bellman function and linear dimension-free estimates in a theorem of Bakry}
 
\author[Carbonaro]{Andrea Carbonaro}

\author[Dragi\v{c}evi\'c]{Oliver Dragi\v{c}evi\'c$^{1}$
}

\address{Andrea Carbonaro\\ University of Birmingham\\ School of Mathematics\\ Watson Building\\ Edgbaston\\ Birmingham B15 2TT\\ United Kingdom}
\curraddr{Dipartimento di Matematica\\ Universit\`a degli Studi di Genova\\ Via Dodecaneso 35, 16146 Genova, Italy}
\email{carbonaro@dima.unige.it}

\address{Oliver Dragi\v{c}evi\'c
 \\University of Ljubljana\\ Faculty of Mathematics and Physics\\ Jadranska 19, SI-1000 Ljubljana\\ Slovenia}
\email{oliver.dragicevic@fmf.uni-lj.si}



\maketitle
\footnotetext[1]{Partially supported by the Ministry of Higher Education,
Science and Technology of Slovenia (research program Analysis and
Geometry, contract no. P1-0291).}

\begin{abstract}
By using an explicit Bellman function, we prove a bilinear embedding theorem for the Laplacian associated with a weighted Riemannian manifold $(M,\mu_\varphi)$ having the Bakry-Emery curvature bounded from below. The embedding, acting on the Cartesian product of $L^p(M,\mu_\varphi)$ and $L^q(T^*M,\mu_\varphi)$, $1/p+1/q=1$, involves estimates which are independent of the dimension of the manifold and linear in $p$. As a consequence we obtain linear dimension-free estimates of the $L^p$ norms of the corresponding shifted Riesz transform. All our proofs are analytic.
\end{abstract}

\section{Introduction}
\label{intro}
Consider a complete Riemannian manifold $(M,g,\mu_0)$ with Riemannian metric $g$ and Riemannian measure $\mu_0$. Let $\wrt$, $\nabla$, ${\rm Grad}$ and $\Delta$ denote, respectively, the exterior and the covariant derivative, the gradient, and the nonnegative Laplace-Beltrami operator on $M$. Given $\varphi\in C^\infty(M)$, consider the weighted measure on $M$ defined by
$$
\wrt \mu_\varphi(x)=e^{-\varphi(x)}\wrt \mu_0(x),
$$
and denote by $\cL_\varphi$ the nonnegative weighted Laplacian defined on $C^\infty_c(M)$ by
$$
\cL_\varphi f=\Delta f+\wrt f(\grad{\varphi} ).
$$
It was proved in \cite{Bakry1,S} that $\cL_\varphi$ is essentially self-adjoint on $L^2(M,\mu_\varphi)$, and with an abuse of notation we still denote by $\cL_\varphi$ its unique self-adjoint extension. The Bakry-Emery curvature tensor associated with $\cL_\varphi$ is defined by
$$
{\rm Ric}_\varphi={\rm Ric}+\nabla^2\varphi,
$$
where ${\rm Ric}$ denotes the Ricci curvature tensor on $M$. For every $a\in\R$, consider the shifted Riesz transform defined by
$$
\cR_a=\wrt (a^2\cI+\cL_\varphi)^{-1/2}.
$$
The following well-known result was first proven by D. Bakry \cite{Bakry1}.
\begin{theorem}
\label{Bakry}
Suppose that ${\rm Ric}_\varphi\geqslant -a^2g$. Then, for every $p$ in $(1,\infty)$, there is $C(p)>0$ such that
$$
\norm{\cR_af}{L^p(T^*M,\mu_\varphi)}\leqslant C(p)\norm{f}{L^p(M,\mu_\varphi)},
$$
for all $f\in \overline{{\rm R}(a^2\cI+\cL_\varphi)\cap L^p(M,\mu_\varphi)}^{L^p}$.
\end{theorem}
This result has been improved by Li in \cite[Theorem 1.4]{Li}, where the author obtained an explicit upper
estimate, namely,
$C(p)=2(p^*-1)(1+4\|\tau\|_p)$.
Here
$p^*=\max\{p,q\}$, $1/p+1/q=1$,
and $\tau$ is the exit time of the standard 3-dimensional Brownian motion from the unit ball in $\R^3$.
One can determine the asymptotic behaviour of $\nor{\tau}_p$
by means of the distribution function
of $\tau$, which has been calculated by Ciesielski and Taylor \cite{CT}. As a result one quickly computes that
$
\nor{\tau}_p\sim p
$
as $p\rightarrow\infty$.
Thus the estimate in \cite{Li} is quadratic in $p$ for $a\geqslant0$, except in the case $a=0$, when the author showed that it suffices to take $C(p)=2(p^*-1)$. A further improvement was made by the same author in \cite[Theorems 1.5, 1.6]{Li2} by demonstrating that, if $a>0$, one can take $C(p)=2(p^*-1)^{3/2}$. 

While the first version of this paper was under review, it was discovered by Ba\~nuelos and Baudoin \cite[Remark 2.1]{BB} that the above cited papers of Li actually contain a gap, owing to which the proofs of his explicit estimates in terms of $p$ are not correct. This gap was addressed by Li himself \cite{Li3, Li4}, while the results originally claimed by Li have recently been proven by Ba\~nuelos and Os\c ekowski \cite{BO}.

The same papers by Li \cite{Li,Li2} and Ba\~nuelos and Os\c ekowski \cite{BO} also contain a thorough review of numerous earlier results about Riesz transforms on various classes of Riemannian manifolds, as well as several applications that further motivate the pursuit of the dimension-free boundedness of Riesz transforms in such generality.

The proofs in \cite{Bakry1,Li,Li2,BO} are probabilistic, and in \cite[p. 269]{Li} the author specifically raises the question of finding an analytic proof of Theorem~\ref{Bakry}. The main objective of this paper is to give a short analytic proof of Bakry's result with explicit linear estimates in $p$ (see Corollary~\ref{c: 2}). We accomplish this by employing the technique of Bellman functions. It originates in stochastic optimal control, while it was brought into harmonic analysis
by Nazarov, Treil and Volberg in the 1994 preprint version of their paper \cite{NTV}.
Here we will follow the scheme laid out in papers \cite{DV,DV-Sch,DV-Kato}. Accordingly, the result in question will be a corollary of the so-called bilinear embedding theorem for the weighted Laplacian $\cL_\varphi$ on $M$ (see Theorem~\ref{t: 3}).
We are able to make the passage from the embedding theorem to the Riesz transforms without using spectral multipliers. This is in contrast with \cite{DV} and \cite{DV-Sch}, although one of the results there (dimension-free estimates of Riesz transforms associated with the Ornstein-Uhlenbeck operator, see \cite{DV}) is a particular case of Corollary \ref{c: 2}. In this light our method can be viewed as an improvement over \cite{DV}.

As remarked above, unlike in Li \cite{Li, Li2} and Ba\~nuelos and Os\c ekowski \cite{BO} our proofs are purely analytic. Both our estimate in Corollary \ref{c: 2} and the one in \cite{BO} are linear in $p$. However, the estimate in \cite{BO} exhibits a smaller numerical constant. A better numerical constant in our theorem could be obtained by using a ``sharper'' Bellman function. Such a function does exist, see \cite{DV} and \cite{VaVo}, but since it is not explicit, it seems much more difficult to work with. The advantage of the Bellman function we utilise in this paper (and which originated in the work of Nazarov and Treil \cite{NT}) is its simplicity and the fact that it admits satisfactory estimates of its partial derivatives (see Theorem \ref{bejaz}).

The same Nazarov-Treil Bellman function was recently used by A. Volberg and the second author \cite{DV-Sch, DV-Kato} in order to obtain similar results for (generalised) Schr\"odinger operators with nonnegative potentials, again yielding dimension-free estimates with sharp (linear) estimates in $p$ involving explicit constants.

Since the introduction of the Bellman function method in harmonic analysis by Nazarov, Treil and Volberg in mid-90's, there has been a whole series of (sharp) inequalities treated with great success by this method. Yet to the best of our knowledge this paper is the very first case of applying Bellman functions on general manifolds rather than on Euclidean spaces. As such it may open a path for a wide range of similar applications in the future. For example, we are currently studying $L^p$ spectral multipliers on weighted Riemannian manifolds by using Bellman functions techniques. This will be contained in a forthcoming paper.

We proceed to the formulation of our main results. Before we can do that precisely, we need to recall a few additional well-known notions and facts.

\section{Preliminaries}
For each $x$ in $M$, we denote the tangent and the cotangent spaces at $x$ respectively by $T_xM$ and $T^*_xM$. For every $j,k\in\N$, we set
$$
T^{j,k}_xM=\underbrace{T_xM\otimes\cdots\otimes T_xM}_{j\ times}\otimes \underbrace{T^*_xM\otimes\cdots  \otimes T^*_xM}_{k\ times},
$$
and we denote by $T^{j,k}M$ the fiber bundle over $M$ whose fibre at $x$ is $T^{j,k}_xM$. A tensor of type $(j,k)$ is just a section of $T^{j,k}M$. We denote the space of smooth tensors of type $(j,k)$ by $C^\infty(T^{j,k}M)$ and identify functions on $M$ with tensors of type $(0,0)$. For each $k=0,\hdots,\dim M$, let $\wedge^k T^*M$ denote the bundle of {\it alternating tensors} of type $(0,k)$, also referred to as $k$-forms. Recall that for every $j,k\in\N$ and $x\in M$ the Riemannian scalar product on $T_xM$ induces a scalar product $\langle\cdot,\cdot\rangle_{T^{j,k}_xM}$ on $T^{j,k}_xM$; this clearly induces a scalar product on $\wedge^kT^*_xM$, for all $k=0,\hdots,\dim M $. We set $\mod{\cdot}^2_{T^{j,k}_xM}=\langle\cdot,\cdot\rangle_{T^{j,k}_xM}$. For each $p\in [1,\infty]$ and $j,k\in\N$, let $L^p(T^{j,k}M,\mu_\varphi)$ be the Banach space of all measurable tensors $u$ of type $(j,k)$  with
$$
\|u\|_{L^p(T^{j,k}M,\mu_\varphi)}=
\begin{cases}
    \big(\int_M \mod{u(x)}_{T^{j,k}_xM}^p\ \wrt\mu_\varphi(x)\big)^{\frac{1}{p}}<\infty, &{\rm if}\quad p\in[1,\infty); \\
    \esss_{x\in M} \mod{u(x)}_{T^{j,k}_xM}<\infty, &{\rm if}\quad p=\infty.
  \end{cases}
$$
When there will be no ambiguity, we shall denote $\mod{\cdot}_{T^{j,k}_xM}$ simply by $|\cdot|$, and $L^p(T^{j,k}M,\mu_\varphi)$ by $L^p(\mu_\varphi)$. If $A$ is an operator on $L^2(\mu_\varphi)$, we denote respectively by ${\rm R}(A)$ and ${\rm N}(A)$ its range and null-space.\\

Furthermore, let
$$
\wrt:  C^\infty(\wedge^k T^*M)\rightarrow C^\infty(\wedge^{k+1}T^*M)\ \ \  {\rm and}\ \ \  \nabla:C^\infty(T^{j,k}M)\rightarrow
C^\infty(T^{j,k+1}M)
$$
be the exterior and the total covariant derivative, respectively, and $\wrt^*_\varphi$ and $\nabla^*_\varphi$ their adjoints on $L^2(\mu_\varphi)$. Recall that on functions $\wrt$ and $\nabla$ coincide with the differential, $\wrt^2=0$, and, for every $u\in T^{j,k}M$, $\eta_i\in C^\infty(T^*M)$ and $X,Y_j\in C^\infty(TM)$, $\nabla u(\eta_1,\dots,\eta_r,X,Y_1,\dots,Y_s)=\nabla_X u(\eta_1,\dots,\eta_r,Y_1,\dots,Y_s)$, where $\nabla_Xu\in T^{j,k}M$ denotes the covariant derivative of $u$ with respect to $X$. Given a system of local coordinates $(x^1,\dots,x^n)$, we set $\nabla_i=\nabla_{\partial_{x^i}}$, $i=1,\dots,n$.

An easy computation gives
$$
\wrt^*_\varphi=\wrt^*_0+i_{\grad{\varphi}},
$$
where $i_{\grad{\varphi}}$ denotes the inner multiplication by $\grad{\varphi}$ on $\wedge^{k+1} T^*M$. The (nonnegative) weighted Hodge-De Rham Laplacian acting on $k$-forms is defined by
$$
\square_{k,\varphi}=\wrt\wrt^*_\varphi+\wrt^*_\varphi\wrt.
$$
It is well known that $\square_{k,\varphi}$, initially defined on smooth $k$-forms with compact support, is essentially self-adjoint on $L^2(\wedge^kT^*M,\mu_\varphi)$ (see \cite{S}). Note that $\square_{0,\varphi}= \cL_\varphi$, and by the Bochner-Weitzenb\"ock formula we have
$$
\square_{1,\varphi}\omega=\nabla^*_0\nabla\omega+\nabla_{{\rm Grad}(\varphi)}\omega+{\rm Ric}_\varphi(\cdot,\sharp\omega),
$$
where $\sharp:T^*_xM\rightarrow T_xM$ is the duality defined by $\omega(X)=\sk{\sharp\omega}{X}_{TM}$ for all $\omega\in T^*_xM$ and $X\in T_xM$ \cite{pigola, berger}.

We set $\vec{\cL_\varphi}=\square_{1,\varphi}$, $P^a_t=\exp(-t(a^2\cI+\cL_\varphi)^{1/2})$ and $\vec{P^a_t}=\exp(-t(a^2\cI+\vec{\cL_\varphi})^{1/2})$. Note that $L^2(M,\mu_\varphi)=\overline{{\rm R}(a^2\cI+\cL_\varphi)}\oplus {\rm N}(a^2\cI+\cL_\varphi)$ where the sum is orthogonal. The Riesz transform $\cR_a$ initially defined on ${\rm R}(a^2\cI+\cL_\varphi)$ extends to a contraction
$$
\cR_a:\overline{{\rm R}(a^2\cI+\cL_\varphi)}\longrightarrow L^2(T^*M).
$$
Note that if $a>0$ then $\overline{{\rm R}(a^2\cI+\cL_\varphi)}=L^2(M,\mu_\varphi)$. Moreover, $N(\cL_\varphi)\neq \{0\}$ if and only if $\mu_\varphi(M)<\infty$; in this case $N(\cL_\varphi)=\{{\rm constant\ functions\ on}\ M\}$. When $a>0$, $\cR_a$ is often called {\it local Riesz transform}.

\begin{lema}
\label{l: commutativita}
For every $f\in C^\infty_c(M)$, $\omega\in C^\infty_c(T^*M)$, $r\geqslant 1$ and $a\geqslant 0$,
\begin{enumerate}[{\it (a)}]
\item \label{ella} $\wrt \cL_\varphi f=\vec{\cL_\varphi}\wrt f$ and $\wrt^*_\varphi\vec{\cL_\varphi}\omega=\cL_\varphi\wrt^*_\varphi\omega$;
\item \label{fitzgerald} $\wrt P^a_tf=\vec{P^a_t}\wrt f$ and $\wrt^*_\varphi\vec{P^a_t}\omega=P^a_t\wrt^*_\varphi\omega$;
\item \label{charles} $|P^a_tf(x)|^r\leqslant P^a_t|f|^r(x)$.
\end{enumerate}
If also ${\rm Ric}_\varphi\geqslant -a^2g$, then
\begin{enumerate}[{\it (a)}]
\addtocounter{enumi}{3}
\item \label{barkley} $|e^{-t\vec{\cL_\varphi}}\omega(x)|_{T^*_xM}\leqslant e^{ta^2}e^{-t\cL_\varphi}|\omega(x)|_{T^*_xM}$;
\item \label{dr.j} $|\vec P^a_t\omega(x)|^r_{T^*_xM}\leq P^0_t|\omega|^r_{T^*_xM}(x)$.
\end{enumerate}
\end{lema}

\begin{proof}
Items $(a)$, $(b)$ and $(d)$ in the lemma have been proved in \cite[Proposition 1.7]{Bakry1}. Since $\cL_\varphi$ generates a
Markovian semigroup on $(M,\mu_\varphi)$ \cite{Bakry2}, we quickly get
\begin{equation}
\label{e: 1}
|e^{-t(a^2\cI+\cL_\varphi)}f(x)|^r\leqslant e^{-t(a^2\cI+\cL_\varphi)}|f|^r(x)\,.
\end{equation}
Set
$
dm(s)=
(\pi s)^{-1/2}e^{-s}\,ds\,.
$
One readily sees that
\begin{equation}
\label{e: 3}
P^a_t=\int_0^\infty e^{-\frac{t^2}{4s}(a^2\cI+\cL_\varphi)}dm(s)
\end{equation}
in the strong operator topology. Hence \eqref{e: 1} also holds with
$P^a_t$ in place of $e^{-t(a^2\cI+\cL_\varphi)}$, and \eqref{charles} is proved. Similarly, \eqref{dr.j} follows from a combination of the item \eqref{barkley} and the subordination formula \eqref{e: 3}.
\end{proof}

\section{Bilinear embedding theorem and Riesz transforms}
We now state the bilinear embedding theorem which is the principal result of the paper. The proof will be given in Section 5. Denote by $\overline{\nabla}$ the total covariant derivative on $M\times \R_+$. Then, for every $\eta\in C^\infty(T^{j,k}(M\times\R_+))$, $|\overline{\nabla}\eta|=\sqrt{|\nabla\eta|^2+|\nabla_t\eta|^2}$.
\begin{theorem}\label{t: 3}
Suppose that $M$ is a complete Riemannian manifold with ${\rm Ric}_\varphi\geqslant -a^2g$. Then for all $p$ in $(1,\infty)$, $f\in C^\infty_c(M)$ and $\omega\in C^\infty_c(T^*M)$,
\begin{equation*}
\int_0^\infty\int_{M}|\overline{\nabla}P^a_tf(x)||\overline{\nabla}\vec{P^a_t}\omega(x)|\,t\wrt\mu_\varphi(x)\wrt t\leqslant 3(p^*-1)\|f\|_{L^p(M,\mu_\varphi)}\|\omega\|_{L^q(T^*M,\mu_\varphi)}.
\end{equation*}
\end{theorem}
The bilinear embedding theorem implies a dimension-free estimate for the $L^p$ norms of the Riesz transform.
\begin{cor}
\label{c: 2}
Under the above conditions,
\begin{equation*}
\norm{\cR_af}{L^p(T^*M,\mu_\varphi)}\leq 12(p^*-1)\norm{f}{L^p(M,\mu_\varphi)},
\end{equation*}
for all $f\in \overline{{\rm R}(a^2\cI+\cL_\varphi)\cap L^p(M,\mu_\varphi)}^{L^p}$.
\end{cor}
\begin{proof}
We claim that for every $f\in C^\infty_c(M)\cap {\rm R}(a^2\cI+\cL_\varphi)$ and $\omega\in C^\infty_c(T^*M)$ we have that
\begin{equation}\label{eq: riesz}
\int_M\sk{\cR_af(x)}{\omega(x)}\wrt\mu_\varphi(x)=4\int_0^\infty\int_M\Sk{\wrt P^a_tf(x)}{\frac{\wrt}{\wrt t}\vec{P^a_t}\omega(x)}\wrt\mu_\varphi(x)\,t\wrt t\,.
\end{equation}
Assuming the claim \eqref{eq: riesz}, Corollary \ref{c: 2} follows immediately from Theorem \ref{t: 3} and the Cauchy-Schwarz inequality. To prove the claim \eqref{eq: riesz}, consider the function
$$
\varphi(t)=\sk{\vec{P^a_t}\cR_af}{\vec{P^a_t}\omega}_{L^2(\mu_\varphi)}\,.
$$
Since $\sk{\cR_af}{\omega}_{L^2(\mu_\varphi)}=\varphi(0)$, it suffices to show that
\begin{equation}
\label{eq: dschcello1}
\varphi(0)=\int_0^\infty\varphi''(t)\,t\wrt t=4\int_0^\infty\Sk{\wrt P^a_tf}{\frac{\wrt}{\wrt  t}\vec{P^a_t}\omega}_{L^2(\mu_\varphi)}\,t\wrt t\,.
\end{equation}
In order to prove the first equality it is enough to show that both $\f(t)$ and $t\f'(t)$ tend to zero as $t\rightarrow\infty$.
First note that, by Lemma~\ref{l: commutativita},
$
\vec{P^a_t}\cR_af=\cR_aP^a_tf\, .
$
Therefore, by the $L^2$ contractivity of both $\cR_a$ and $\vec{P^a_t}$,
$
|\f(t)|\leqslant \norm{P^a_tf}{L^2(\mu_\varphi)}\norm{\omega}{L^2(\mu_\varphi)}\,.
$
Since $f\in {\rm R}(a^2\cI+\cL_\varphi)$, the spectral theorem gives that $P^a_tf\rightarrow 0$ in $L^2(\mu_\varphi)$ as $t\rightarrow\infty$.

Similarly, Lemma~\ref{l: commutativita} gives
$$
\aligned
\f'(t) & = 2\sk{(a^2\cI+\vec{\cL_\varphi})\vec{P^a_t}\wrt (a^2\cI+\cL_\varphi)^{-1/2}f}{\vec{P^a_t}\omega}_{L^2(\mu_\varphi)}\\
&  =  2\sk{P^a_tf}{P^a_t\wrt^* \omega}_{L^2(\mu_\varphi)}\,,
\endaligned
$$
therefore $\lim_{t\rightarrow\infty}t|\f'(t)|=0$ as before. The second equality in \eqref{eq: dschcello1} can be verified by a straightforward calculation, again with the help of Lemma~\ref{l: commutativita}.
\end{proof}

\begin{remark}
The idea of representing the Riesz transform by using Poisson semigroups on functions and differential forms is certainly not new.
This is  a well-known argument which originates in the work of Bakry \cite[p. 161]{Bakry1} and has been used later by several authors, see for example Coulhon and Duong \cite[Theorem 5.1]{CD} and Li \cite[p. 631]{Li1}. In the special case of the Ornstein-Uhlenbeck operator one can, instead of the Poisson semigroup on differential forms, use the Poisson semigroup on functions paired with certain spectral multipliers \cite{DV}. However in such a case one only gets sharp results depending on whether one is able to give $L^p$ estimates of the corresponding spectral multiplier which are {\it independent both of the dimension and $p$}. This was presented, with different degrees of success, in \cite{DV} and \cite{DV-Sch}.
\end{remark}

\section{Bellman function}
\label{celjustka}
As announced above, the main tool in the proof of the bilinear embedding Theorem~\ref{t: 3} will be a particular Bellman function. Throughout this section we assume that $p\geqslant 2$, $q=p/(p-1)$ and $\delta=q(q-1)/8$ are fixed. Observe that $\delta\sim (p-1)^{-1}$.

Fix $n\in \N$ and define the Bellman function
$
Q: \R\times \R^n\longrightarrow [0,\infty)
$
by setting
\begin{equation*}
\label{sarmat}
Q(\zeta,\eta)=\frac12 \beta(|\zeta|,|\eta|)\,,
\end{equation*}
where
\begin{equation*}
\label{eq:goldberg}
\beta(u,v)=
u^{p}+v^{q}+\delta
\left\{
\aligned
& u^2v^{2-q} & ; & \ \ u^p\leqslant v^q\\
& \frac{2}{p}\,u^{p}+\left(\frac{2}{q}-1\right)v^{q}
& ; &\ \ u^p\geqslant v^q
\endaligned\right.
\end{equation*}
for any $u,v\geqslant 0$. For every $(\zeta,\eta)\in \R\times\R^n$ set $U(\zeta,\eta)=(|\zeta|,|\eta|)$. The function $Q$ belongs to $C^1(\R\times\R^n)$, and it is of order $C^2$ everywhere {\it except} on the set $U^{-1}(\Upsilon_0)$, where
$$
\Upsilon_0=\mn{(u,v)\in[0,\infty)\times[0,\infty)}{(v=0)\vee (u^p=v^q)}\,.
$$
\begin{remark}
The origins of this function lie in the paper of Nazarov and Treil \cite{NT}. A modification of their function was later applied in \cite{DV,DV-Sch}. Here we use a simplified variant which comprises only two variables. It was introduced in \cite{DV-Kato}. The function $Q$ above is the same as in \cite{DV-Kato}, except that it differs by a sign; thus it is nonnegative while the function in \cite{DV-Kato} was nonpositive.
\end{remark}
\begin{remark}
\label{ljudski oblik}
In contrast to \cite{DV,DV-Sch, DV-Kato}, to keep our notation reasonable and to gain some transparency and simplicity in the proofs, we use a Bellman function involving only real variables. This allows us to prove Theorem~\ref{t: 3} and Corollary~\ref{c: 2} just for real-valued functions and differential forms; the corresponding estimates for complex-valued functions and differential forms easily follow by estimating separately the real and imaginary parts. Note that the argument above gives an appropriately bigger constant. In order to preserve the same constants one could readily instead use a ``complex" Bellman function as in \cite{DV,DV-Sch,DV-Kato} and prove Theorem~\ref{t: 3} and Corollary~\ref{c: 2} for complex-valued functions and differential forms.
\end{remark}

Throughout the rest of the paper we shall use the following notation: if $m\in\N$, $\Omega\subset\R^m$ is open, $\Phi\in C^\infty(\Omega)$, $\omega\in\Omega$ and $x\in\R^m$, then we set
$$
H_\Phi(\omega ; x)=\sk{{\rm Hess}(\Phi)_{\omega}x}{x}_{\R^m}\,,
$$
where ${\rm Hess}(\Phi)_{\omega}$ is the Hessian matrix of $\Phi$ at $\omega$, i.e. $[\pd_{x_ix_j}\Phi(\omega)]_{i,j=1}^m$.
\medskip

The following result, essentially proved in \cite{DV-Sch}, summarizes the properties of $Q$.
\begin{theorem}
\label{bejaz}
For every $u,v\geq 0$,
\begin{enumerate}[{\rm (i)}]
\item
\label{anterija}
$0\leqslant \beta(u,v)\leqslant (1+\delta)(u^p+v^q)$.
\end{enumerate}
If $\xi=(\zeta,\eta)\in(\R\times\R^n)\backslash U^{-1}(\Upsilon_0)$, then there exists  $\tau=\tau(|\zeta|,|\eta|)>0$ such that
\begin{enumerate}[{\rm (i)}]
\addtocounter{enumi}{1}
\item
\label{kupres}
$
H_Q(\xi ; w)\geqslant \delta\big(\tau |w_1|^2+\tau^{-1}|w_2|^2\big)\,,
$
for all $w=(w_1,w_2)\in\R\times\R^n$.
\end{enumerate}
Moreover, there is a certain absolute $C=C(p)>0$ such that for every $u,v>0$,
\begin{enumerate}[{\rm (i)}]
\addtocounter{enumi}{2}
\item
\label{brcko}
$0\leqslant\partial_u\beta(u,v)
 \leqslant
C\max\{u^{p-1},v\}
\hskip 15pt\text{and}\hskip 15pt
0\leqslant\partial_v\beta(u,v)
\leqslant
Cv^{q-1}$.
\end{enumerate}
\end{theorem}

As noted earlier, while $Q$ is of class $C^1$, it is not globally $C^2$.
One can fix this in a standard fashion by taking convolutions with mollifiers.
More precisely, denote by $B^{n+1}$ the open unit ball in $\R^{n+1}$ and define
$$
\psi(x)=c_{n+1}e^{-\frac{1}{1-|x|^2}}\chi_{B^{n+1}}(x)\,,
$$
where $c_{n+1}$ is chosen so that the integral of $\psi$ over $\R^{n+1}$ is equal to one. For any $\kappa>0$ and $x\in\R^{n+1}$ set
$$
\psi_\kappa(x)=\frac1{\kappa^{n+1}}\,\psi\Big(\frac{x}{\kappa}\Big).
$$
The (regular) Bellman function $Q_\kappa$ is defined on $\R\times\R^n$ by
\begin{equation*}
\label{barenboim}
Q_\kappa=\psi_\kappa*Q,
\end{equation*}
where $*$ denotes the convolution in $\R^{n+1}$. Since both $Q$ and $\psi_\kappa$ are biradial, there exists $\beta_\kappa:[0,\infty)\times[0,\infty)\rightarrow[0,\infty)$ such that
$$
Q_\kappa(\zeta,\eta)=\frac12\beta_\kappa(|\zeta|,|\eta|)\,,
$$
for all $(\zeta,\eta)\in\R\times\R^n$.

\begin{theorem}
\label{t: regular bellman}
Let $\kappa\in (0,1)$. Then $Q_\kappa\in C^\infty(\R^{n+1})$ and, for any $u,v\geq 0$,
\begin{enumerate}[{\rm (i')}]
\item
\label{anterija'}
$0\leqslant \beta_\kappa(u,v)\leqslant (1+\delta)\big[(u+\kappa)^p+(v+\kappa)^q\big]$.
\end{enumerate}
For any $\xi=(\zeta,\eta)\in \R\times\R^n$,
there exists  $\tau_\kappa=\tau_\kappa(|\zeta|,|\eta|)>0$ such that
\begin{enumerate}[{\rm (i')}]
\addtocounter{enumi}{1}
\item
\label{kupres'}
$
H_{Q_\kappa}(\xi;w)
\geqslant \delta\big(\tau_\kappa |w_1|^2+\tau^{-1}_\kappa|w_2|^2\big)\,,
$
for all $w=(w_1,w_2)\in\R\times\R^n$.
\end{enumerate}
Moreover, there is a certain absolute $C=C(p)>0$  such that for every $u,v\geq0$,
\begin{enumerate}[{\rm (i')}]
\addtocounter{enumi}{2}
\item
\label{brcko'}
$0\leqslant\partial_u \beta_\kappa(u,v)\leqslant C\max\{(u+\kappa)^{p-1}, v+\kappa\}$ and $0\leqslant\partial_v \beta_\kappa(u,v)\leqslant C(v+\kappa)^{q-1}$.
\end{enumerate}
\end{theorem}
\begin{proof}
Properties (\ref{anterija'}') and (\ref{brcko'}') follow from definition of $\beta_\kappa$ and the corresponding properties (\ref{anterija}) and (\ref{brcko}) of $\beta$ in Theorem~\ref{bejaz}. We give a rigorous proof of the estimates for $\pd_v\beta$ which comprise part (iii'). Other inequalities are proven in a very similar way and therefore we will omit their proofs.

\medskip
We start by showing that $\pd_v\beta_\kappa(u,v)\geq 0$ for $u,v\geq 0$. Since $Q_\kappa$ is differentiable and $Q_\kappa(\zeta,\eta)=\beta_\kappa(|\zeta|,|\eta|)$, we have that $\pd_v\beta_\kappa(u,0)=0$ for all $u\geq 0$. Moreover,  for $\zeta\in\R$ and $\eta=(\eta_1,\hdots,\eta_n)\in\R^n$,
\begin{equation}\label{l: 3}
\pd_{\eta_1}Q_\kappa(\zeta,\eta)=\frac{\eta_1}{2|\eta|}\,\pd_v\beta_\kappa(|\zeta|,|\eta|)\,.
\end{equation}
Therefore it suffices to verify that $\pd_{\eta_1}Q_\kappa(\zeta,\eta)\geq 0$ whenever $\eta_1>0$.
By definition,
$$
\pd_{\eta_1}Q_\kappa(\zeta,\eta)
=\int_\R\int_{\R^{n-1}}\int_\R\pd_{\eta_1}Q(\zeta-\zeta',\eta_1-\eta_1',\hat\eta-\hat\eta')\psi_\kappa(\zeta',\eta_1',\hat\eta')\,\wrt\eta_1'\,\wrt\hat\eta'\,\wrt\zeta'\,,
$$
where for $\xi=(\xi_1,\hdots,\xi_n)\in\R^n$ we denote $\hat \xi=(\xi_2,\hdots,\xi_n)\in\R^{n-1}$.
It suffices to prove that the inner integral is positive. To this end, fix $\zeta,\zeta',\hat\eta,\hat\eta'$ and set, for $x\in\R$,
$$
f(x)=\pd_{\eta_1}Q(\zeta-\zeta',x,\hat\eta-\hat\eta')
\hskip 20pt
\text{and}
\hskip 20pt
g(x)=\psi_\kappa(\zeta',x,\hat\eta')\,.
$$
Then the inner integral is precisely $(f*g)(\eta_1)$.
From the properties of $Q$ and $\psi_\kappa$ it emerges that the function $f$ is odd and positive on $\R_+$, while $g$ is even and decreasing on $\R_+$.
It is only left to apply the next lemma and the desired positivity will follow.

\begin{lema}
Let $f,g :\R\rightarrow\R$ be continuous functions. Suppose $f$ is odd and nonnegative 
on $\R_+$, while $g$ is even, compactly supported 
and decreasing on $\R_+$. Then $f*g \geqslant 0$ on $\R_+$.
\end{lema}
\begin{proof}
Since $f$ is odd and $g $ even, we can write, for $x\geqslant 0$,
$$
(f*g )(x)=\int_0^xf(y)[g (x-y)-g (x+y)]\,\wrt y+\int_x^\infty f(y)[g (y-x)-g (y+x)]\,\wrt y\,.
$$
Now use the remaining assumptions on $f$ and $g$.
\end{proof}
Next we prove the upper estimates for $\partial_v\beta_\kappa$ in (iii'). Take $u,v\geq 0$ and $\kappa>0$. We would like to show that $|\pd_v\beta_\kappa(u,v)|\leqslant C (v+\kappa)^{q-1}$. We apply \eqref{l: 3} with $\zeta=u$ and $\eta=(v,0,\hdots,0)\in\R^n$. Since $\psi_\kappa$ is nonnegative, it follows that
$$
|\pd_v\beta_\kappa(u,v)|\leqslant 2\int|\pd_{\eta_1} Q(\zeta-\zeta',\eta-\eta')|\psi_\kappa(\zeta',\eta')\,\wrt\zeta'\,\wrt\eta'\,.
$$
By Theorem 3 (iii),
$
|\pd_{\eta_1} Q(\zeta-\zeta',\eta-\eta')|\leqslant C|\eta-\eta'|^{q-1}\,.
$
Finally use that $\psi_\kappa$ is supported in the ball $B(0,\kappa)$
and $\int\psi_\kappa=1$, which yields
$$
|\pd_{v}\beta_\kappa(u,v)|
\leqslant C\left(|\eta|+\kappa\right)^{q-1}
=C\left(v+\kappa\right)^{q-1}
\,,
$$
as desired.

\medskip
We now prove (\ref{kupres'}'). First notice that the second-order distributional derivatives of $Q$ exist and coincide almost everywhere with the usual ones. This is the case because $Q$ belongs to $C^1(\R\times\R^n)$, its second-order partial derivatives exist in $\R^{n+1}\setminus U^{-1}(\Upsilon_0)$, and are locally integrable in $\R^{n+1}$. Consequently, if we set $\xi=(\zeta,\eta)\in \R\times\R^n$ and $w=(w_1,w_2)\in \R\times\R^n$, then we have
$$
H_{Q_\kappa}(\xi;w)
=\int
H_{Q}(\xi-y;w)
\,\psi_\kappa(y)\,dy\,.
$$
By Theorem \ref{bejaz} \eqref{kupres}, the first factor inside the integral is almost everywhere bounded from below by
$
\delta\big(\tau|w_1|^2+\tau^{-1}|w_2|^2\big)\,,
$
where $\tau$ is a function of $\xi-y$. Consequently,
$$
H_{Q_\kappa}(\xi;w)
\geqslant
\delta\big((\tau*\psi_\kappa)(\xi)|w_1|^2+(\tau^{-1}*\psi_\kappa)(\xi)|w_2|^2\big)\,.
$$
Notice that H\"older's inequality gives
$$
(\tau*\psi_\kappa)(\xi)(\tau^{-1}*\psi_\kappa)(\xi)\geqslant
\bigg[
\int
\sqrt{\tau(y)\psi_\kappa(\xi-y)}\,\sqrt{\tau^{-1}(y)\psi_\kappa(\xi-y)}\,dy\bigg]^2=1\,.
$$
It follows that
\begin{equation*}
\label{aint}
H_{Q_\kappa}(\xi;w)
\geqslant
\delta\big(\tau_\kappa|w_1|^2+\tau_\kappa^{-1}|w_2|^2\big)\,,
\end{equation*}
where $\tau_\kappa=\tau*\psi_\kappa$.
\end{proof}

The fact that the $Q_\kappa$'s are radial functions allows us to define Bellman functions on manifolds.

\begin{definition}
For every $\kappa>0$, the regular Bellman function
$$
\widetilde{Q}_\kappa:\R\times T^*M\rightarrow [0,\infty)
$$
is defined on each fiber by the rule
$$
\widetilde{Q}_{\kappa}(\zeta,\eta)=\frac{1}{2}\beta_\kappa(|\zeta|,|\eta|_{T^*_xM}),\ \ \  (\zeta,\eta)\in\R\times T^*_xM.
$$
\end{definition}
We conclude this section with a technical result that will be used in the proof of Theorem~\ref{t: 3}. Consider the operators
$$
\cL^{'}_\varphi=\cL_\varphi-\frac{\partial^2}{\partial t^2}
\hskip 20pt
\text{and}
\hskip 20pt
\vec{\cL^{'}}_\varphi=\vec{\cL_\varphi}-\frac{\partial^2}{\partial t^2},
$$
and denote by $E$ the fiber bundle over $M\times \R_+$ whose fiber at $(x,t)$ is $T^*_xM$.

\begin{lema}\label{l: 1}
Suppose that $\zeta\in C^\infty(M\times\R_+)$, $\eta\in C^\infty(E)$, and define the function $F:M\times\R_+\rightarrow \R$ by the rule
$$
\aligned
F=-\cL^{'}_\varphi \widetilde{Q}_\kappa(\zeta,\eta)
& +\frac{\partial_u{\beta}_{\kappa}(|\zeta|,|\eta|)}{2|\zeta|}\,\zeta\cL^{'}_\varphi\zeta\\
& +\frac{\partial_v{\beta}_{\kappa}(|\zeta|,|\eta|)}{2|\eta|}\left(\langle\vec{\cL_\varphi^{'}}\eta,\eta\rangle-{\rm Ric}_\varphi(\sharp \eta,\sharp \eta)\right).
\endaligned
$$
Then we have
$$
F\geqslant 2\delta|\overline{\nabla}\zeta||\overline{\nabla}\eta|,
$$
where $\overline{\nabla}$ denotes the total covariant derivative on $M\times\R_+$.
\end{lema}
\begin{proof} The lemma follows (by direct computation in exponential local coordinates) from the very definition of $\widetilde{Q}_\kappa$, the Bochner formula \cite[eq. (0.3)]{Bakry1} and Theorem~\ref{t: regular bellman} ({\rm ii'}). For the reader's convenience, we include a full proof.

\medskip
First note that
\begin{align}\label{eq: 9}
&|\overline{\nabla}\zeta(x,t)|^2_{T^*_{(x,t)}(M\times\R_+)}=|\nabla\zeta(x,t)|^2_{T^{*}_xM}+|\partial_t\zeta(x,t)|^2,\\
&|\overline{\nabla}\eta(x,t)|^2_{T^{0,2}_{(x,t)}(M\times\R_+)}=|\nabla\eta(x,t)|^2_{T^{0,2}_xM}+|\partial_t\eta(x,t)|^2_{T^*_xM}.\nonumber
\end{align}
In order to simplify the calculations, we set
$$
\widetilde \beta(u,v)=\frac{1}{2}\beta_\kappa(u^{1/2},v^{1/2}),
$$
so that
$$
\widetilde Q_\kappa(\zeta(x,t),\eta(x,t))=\widetilde\beta(|\zeta(x,t)|^2,|\eta(x,t)|^2_{T^*_xM}).
$$
We now write the function $F$ in terms of $|\zeta|$, $|\eta|$, $|\nabla\zeta|^2$, $|\partial_t\zeta|^2$,  $|\nabla\eta|^2$, $|\partial\eta|^2$, $\wrt |\zeta|^2$, $\partial_t|\zeta|^2$, $\wrt |\eta|^2$, $\partial_t |\eta|^2$ and the partial derivatives of $\widetilde\beta$.

\medskip
By the very definition of $\cL_\varphi$,
$$
\aligned
-\cL_\varphi\widetilde{Q}_\kappa=&-\partial_u\widetilde\beta\cdot\cL_\varphi|\zeta|^2-\partial_v\widetilde\beta\cdot\cL_\varphi|\eta|^2\\
&+\partial^2_{uu}\widetilde\beta\cdot\langle\wrt |\zeta|^2,\wrt |\zeta|^2\rangle+2\partial^2_{uv}\widetilde\beta\cdot\langle\wrt |\zeta|^2,\wrt |\eta|^2\rangle\\
& +\partial^2_{vv}\widetilde\beta\cdot\langle\wrt |\eta|^2,\wrt |\eta|^2\rangle,
\endaligned
$$
and
$$
\aligned
\partial^2_{tt}\widetilde{Q}_\kappa=&\hskip 7pt \partial_u\widetilde\beta\cdot\partial^2_{tt}|\zeta|^2+\partial_v\widetilde\beta\cdot\partial^2_{tt}|\eta|^2\\
&+\partial^2_{uu}\widetilde\beta\cdot\langle\partial_t |\zeta|^2,\partial_t |\zeta|^2\rangle+2\partial^2_{uv}\widetilde\beta\cdot\langle\partial_t |\zeta|^2,\partial_t |\eta|^2\rangle\\
&+\partial^2_{vv}\widetilde\beta\cdot\langle\partial_t |\eta|^2,\partial_t |\eta|^2\rangle.
\endaligned
$$
Moreover,
$$
-\cL_\varphi|\zeta|^2=2|\nabla\zeta|^2-2\zeta\cL_\varphi\zeta
$$
and by the Bochner formula \cite[eq. (0.3)]{Bakry1},
\begin{equation*}
-\cL_\varphi|\eta|^2=2|\nabla\eta|^2-2\langle\vec{\cL_\varphi}\eta,\eta\rangle+2{\rm Ric}_\varphi(\sharp\eta,\sharp\eta).
\end{equation*}
Since $\zeta\in C^\infty(M\times\R_+)$ and $\eta\in C^\infty(E)$,
$$
\aligned
\partial^2_{tt}|\zeta|^2& =2|\partial_t\zeta|^2+2\partial^2_{tt}\zeta\cdot \zeta\,,\\
\partial^2_{tt}|\eta|^2& =2|\partial_t\eta|^2+2\langle\partial^2_{tt}\eta,\eta\rangle.
\endaligned
$$
It follows that
\begin{align}\label{eq: 7}
F=&\left[2\partial_u\widetilde\beta\cdot |\nabla\zeta|^2+\partial^2_{uu}\widetilde\beta\cdot \langle\wrt|\zeta|^2,\wrt|\zeta|^2\rangle\right]\nonumber\\
&+2\partial^2_{uv}\widetilde\beta\cdot \langle\wrt|\zeta|^2,\wrt|\eta|^2\rangle\nonumber\\
&+\left[2\partial_v\widetilde\beta\cdot |\nabla\eta|^2+\partial^2_{vv}\widetilde\beta\cdot \langle\wrt|\eta|^2,\wrt|\eta|^2\rangle\right]\\
&+\left[2\partial_u\widetilde\beta\cdot |\partial_t\zeta|^2+\partial^2_{uu}\widetilde\beta\cdot \langle\partial_t|\zeta|^2,\partial_t|\zeta|^2\rangle\right]\nonumber\\
&+2\partial^2_{uv}\widetilde\beta\cdot \langle\partial_t|\zeta|^2,\partial_t|\eta|^2\rangle\nonumber\\
&+\left[2\partial_v\widetilde\beta\cdot |\partial_t\eta|^2+\partial^2_{vv}\widetilde\beta\cdot \langle\partial_t|\eta|^2,\partial_t|\eta|^2\rangle\right].\nonumber
\end{align}
We now verify the inequality
$$
 F(x,t)\geqslant 2\delta|\overline{\nabla}\zeta(x,t)||\overline{\nabla}\eta(x,t)|
$$
at any point $(x,t)\in M\times \R_+$.
Fix $(x,t)\in M\times \R_+$ and exponential local coordinates $\overline{x}=(x^1,\dots x^n)$ centered at $x$.
Then locally, $g^{-1}=(g^{ij}(\overline{x}))$, $\zeta=\widetilde\zeta(\overline{x},t)$, $\eta=\tilde\eta_1(\overline{x},t)\wrt x^1+\cdots+\tilde\eta_n(\overline{x},t)\wrt x^n$, $\zeta(x,t)=\tilde\zeta(0,t)$ and $\eta(x,t)=\tilde \eta_1(0,t)\wrt x^1_{|_{\overline{x}=0}},\dots,\tilde \eta_n(0,t)\wrt x^n_{|_{\overline{x}=0}}$.

Since $g^{ij}(0)=\delta_{ij}$ and the Christoffel symbols satisfy $\Gamma^{ij}_k(0)=0$, we have that
\begin{align}
\label{eq: 8}
| \eta(x,t)|^2_{T^*_xM}& =\sum_{i}|\tilde\eta_i(0,t)|^2,\nonumber\\
|\nabla \eta(x,t)|^2_{T^{0,2}_xM}& =\sum_{k,i}|\partial_k\tilde\eta_i(0,t)|^2,\\
|\nabla \zeta(x,t)|^2_{T^*_xM}& =\sum_{k}|\partial_k\tilde\zeta(0,t)|^2.\nonumber
\end{align}
Moreover,
$$
\aligned
&\langle\wrt|\eta|^2(x,t),\wrt|\eta|^2(x,t)\rangle_{T^*_xM}=4\sum_k\sum_{i,j}\tilde\eta_i(0,t)\tilde\eta_j(0,t)
\partial_k\tilde\eta_i(0,t)\partial_k\tilde\eta_j(0,t),\\
&\langle\wrt|\zeta|^2(x,t),\wrt|\zeta|^2(x,t)\rangle_{T^*_xM}=4\sum_k\tilde\zeta(0,t)^2(\partial_k\tilde\zeta(0,t))^2.
\endaligned
$$
Define $\tilde\eta(0,t)=(\tilde\eta_1(0,t),\dots,\tilde\eta_n(0,t))$. By the identity
$$
Q_\kappa(\zeta,\eta)=\frac{1}{2}\beta_\kappa(|\zeta|,|\eta|)=\widetilde\beta (\zeta^2,\eta^2_1+\cdots +\eta^2_n)
\hskip 30pt\text{for }\zeta\in\R, \eta\in\R^n,
$$
we obtain
$$
\aligned
&2\partial_u\widetilde\beta(|\zeta(x,t)|^2,|\eta(x,t)|^2)\cdot |\nabla\zeta(x,t)|^2\\
&+\partial^2_{uu}\widetilde\beta(|\zeta(x,t)|^2,|\eta(x,t)|^2)\cdot \langle\wrt|\zeta(x,t)|^2,\wrt|\zeta(x,t)|^2\rangle\\
&\hskip 40pt=\sum_k\partial^2_{\zeta\zeta}Q_\kappa(\tilde\zeta(0,t),\tilde\eta(0,t))(\partial_k\tilde\zeta(0,t))^2,
\endaligned
$$
$$
\aligned
&\partial^2_{uv}\widetilde\beta(|\zeta(x,t)|^2,|\eta(x,t)|^2)\cdot\langle\wrt|\zeta(x,t)|^2,\wrt|\eta(x,t)|^2\rangle\\
&\hskip 40pt=\sum_k\sum_i\partial^2_{\zeta\eta_i}Q_\kappa(\tilde\zeta(0,t),\tilde\eta(0,t))\partial_k\tilde\zeta(0,t)\partial_k\tilde\eta_i(0,t),
\endaligned
$$
and
$$
\aligned
&2\partial_v\widetilde\beta(|\zeta(x,t)|^2,|\eta(x,t)|^2)\cdot |\nabla\eta(x,t)|^2\\
&+\partial^2_{vv}\widetilde\beta(|\zeta(x,t)|^2,|\eta(x,t)|^2)\cdot\langle\wrt|\eta(x,t)|^2,\wrt|\eta(x,t)|^2\rangle\\
&\hskip 40pt=\sum_k\sum_{i,j}\partial^2_{\eta_i\eta_j}Q_\kappa(\tilde\zeta(0,t),\tilde\eta(0,t))\partial_k\tilde\eta_i(0,t)\partial_k\tilde\eta_j(0,t).
\endaligned
$$
A similar computation gives
$$
\aligned
&2\partial_u\widetilde\beta(|\zeta(x,t)|^2,|\eta(x,t)|^2)\cdot |\partial_t\zeta(x,t)|^2\\
&+\partial^2_{uu}\widetilde\beta(|\zeta(x,t)|^2,|\eta(x,t)|^2)\cdot \langle\partial_t|\zeta(x,t)|^2,\partial_t|\zeta(x,t)|^2\rangle\\
&\hskip 40pt=\partial^2_{\zeta\zeta}Q_\kappa(\tilde\zeta(0,t),\tilde\eta(0,t))(\partial_t\tilde\zeta(0,t))^2,
\endaligned
$$

$$
\aligned
&2\partial_v\widetilde\beta(|\zeta(x,t)|^2,|\eta(x,t)|^2)\cdot |\partial_t\eta(x,t)|^2\\
&+\partial^2_{vv}\widetilde\beta(|\zeta(x,t)|^2,|\eta(x,t)|^2)\cdot\langle\partial_t|\eta(x,t)|^2,\partial_t|\eta(x,t)|^2\rangle\\
&\hskip 40pt=\sum_{i,j}\partial^2_{\eta_i\eta_j}Q_\kappa(\tilde\zeta(0,t),\tilde\eta(0,t))\partial_t\tilde\eta_i(0,t)\partial_t\tilde\eta_j(0,t)
\endaligned
$$
and
$$
\aligned
&\partial^2_{uv}\widetilde\beta(|\zeta(x,t)|^2,|\eta(x,t)|^2)\cdot\langle\partial_t|\zeta(x,t)|^2,\partial_t|\eta(x,t)|^2\rangle\\
&\hskip 40pt=\sum_i\partial^2_{\zeta\eta_i}Q_\kappa(\tilde\zeta(0,t),\tilde\eta(0,t))\partial_t\tilde\zeta(0,t)\partial_t\tilde\eta_i(0,t).
\endaligned
$$
It follows from \eqref{eq: 7} that
$$
F(x,t)=\sum^{n}_{k=0}H_{Q_\kappa}\left(\big(\tilde\zeta(0,t),\tilde\eta(0,t)\big);\big(\partial_k\tilde\zeta(0,t),\partial_k\tilde\eta(0,t)\big)\right),
$$
where $\partial_{0}=\partial_t$. Hence, by Theorem~\ref{t: regular bellman} ({\rm ii}'), \eqref{eq: 8} and \eqref{eq: 9},
$$
\aligned
F(x,t)&\geqslant 2\delta\sqrt{\sum^{n}_{k=0}|\partial_k\tilde\zeta(0,t)|^2}\sqrt{\sum^{n}_{k=0}|\partial_k\tilde\eta(0,t)|^2_{\R^n}}\\
&=2\delta|\overline{\nabla}\zeta(x,t)||\overline{\nabla}\eta(x,t)|,
\endaligned
$$
as required.
\end{proof}

\section{Proof of Theorem \ref{t: 3}}

We first prove the theorem for $p\geqslant 2$. Let $f\in C^\infty_c(M)$ and $\omega\in C^\infty_c(T^*M)$.
In view of Remark \ref{ljudski oblik} we can assume $f,\omega$ to be real-valued. Fix $o\in M$ and $\epsilon>0$. For every $s,l>0$, define $K_{s,l}=\overline{B(o,2l)}\times[1/s,s]$ and
\begin{equation*}
\label{eq: 1}
\kappa_{s,l}=\epsilon\inf_{(x,t)\in K_{s,l}}\min\{P^a_{t}|f|(x)\ ,\ P^0_{t}|\omega|(x)\}.
\end{equation*}
Since $P^a_t$ is an integral operator with positive kernel, and $(x,t)\mapsto P^a_tu(x)$ is continuous for all nice $u$, it follows that $\kappa_{s,l}>0$. Next define the function $b_{s,l}$ by setting
$$
b_{s,l}(x,t)=\widetilde{Q}_{\kappa_{s,l}}(P^a_tf(x),\vec P^a_t\omega(x)),
$$
for all $(x,t)\in M\times\R_+$.

Similar to the Euclidean case \cite{DV-Sch} the bulk of the proof of Theorem \ref{t: 3} will consist of estimating an integral involving $\cL^{'}_\varphi b_{s,l}$ from below and above. This will be the content of Propositions \ref{p: 2} and \ref{p: 3}, respectively.

\begin{proposition}
\label{p: 2}
Suppose that ${\rm Ric}_\varphi\geqslant -a^2g$. Then, for all $(x,t)\in M\times\R_+$,
$$
-\cL^{'}_\varphi b_{s,l}(x,t)\geqslant 2\delta|\overline{\nabla}P^a_tf(x)||\overline{\nabla}\vec{P^a_t}\omega(x)|\,.
$$
\end{proposition}
\begin{proof}
We apply Lemma~\ref{l: 1} with $\kappa=\kappa_{s,l}$, $\zeta=P^a_tf$ and $\eta=\vec P^a_t\omega$. Since $\cL^{'}_\varphi P^a_tf=-a^2P^a_tf$, $\vec{ \cL^{'}}_\varphi \vec P^a_t\omega=-a^2\vec P^a_t\omega$, ${\rm Ric}_\varphi(\sharp \vec P^a_t\omega,\sharp \vec P^a_t\omega)\geqslant -a^2|\vec P^a_t\omega|^2$ and, by Theorem~\ref{t: regular bellman} (\ref{brcko'}'), the first-order partial derivatives of $\beta_{\kappa_{s,l}}$ are nonnegative,
$$
-\cL^{'}_\varphi b_{s,l}(x,t)\geqslant F(x,t)\geqslant 2\delta|\overline{\nabla}P^a_tf(x)||\overline{\nabla}\vec{P^a_t}\omega(x)|,
$$
which is the statement from the proposition.
\end{proof}
In order to estimate $-\cL^{'}_{\varphi}b_{s,l}$ from above we need a preliminary result.
\begin{lema}
\label{l: 2}
Suppose that ${\rm Ric}_\varphi\geq -a^2g$. Then, for every $(x,t)\in K_{s,l}$,
\begin{equation*}
\label{eq: 2}
b_{s,l}(x,t)\leq \frac{1+\delta}2\,(1+\epsilon)^{p}\left(P^a_t|f|^p(x)+P^0_t|\omega|^q(x)\right)\,.
\end{equation*}
Moreover, there exists $C=C(\epsilon,p)$ such that, for every $(x,t)\in K_{s,l}$,
\begin{equation*}
\label{eq: 3}
\aligned
\mod{\partial_tb_{s,l}(x,t)}\leqslant C\Big(&\max\{(P^a_t|f|(x))^{p-1},P^0_t|\omega|(x)\}\mod{\partial_t P^a_tf(x)}\\
&+(P^0_t|\omega|(x))^{q-1}|\partial_t \vec P^a_t\omega(x)|\Big)\,.
\endaligned
\end{equation*}
\end{lema}
\begin{proof}
By combining Theorem~\ref{t: regular bellman} (\ref{anterija'}') with Lemma~\ref{l: commutativita} \eqref{charles} and \eqref{dr.j}, we get
 $$
 b_{s,l}(x,t)\leq \frac{1+\delta}2\,\Big[(P^a_t|f|(x)+\kappa_{s,l})^p+(P^0_t|\omega|(x)+\kappa_{s,l})^q\Big]\,.
 $$
The first part of the lemma now follows from the definition of $\kappa_{s,l}$ and Lemma~\ref{l: commutativita} \eqref{charles}, \eqref{dr.j}.

Observe that
$$
\aligned
2|\partial_tb_{s,l}(x,t)|
\leqslant& \hskip 12pt\partial_u\beta_{\kappa_{s,l}}(|P^a_tf(x)|,|\vec P^a_t\omega(x)|)|\partial_t P^a_tf(x)|\
\\
&+\partial_v\beta_{\kappa_{s,l}}(|P^a_tf(x)|,|\vec P^a_t\omega(x)|)|\partial_t \vec P^a_t\omega(x)|.
\endaligned
$$
The second part of the lemma follows from the definition of $\kappa_{s,l}$, by combining the above inequality with Theorem~\ref{t: regular bellman} (\ref{brcko'}') and Lemma~\ref{l: commutativita} \eqref{charles} and \eqref{dr.j}.
\end{proof}

\begin{proposition}\label{p: 3}
Suppose that ${\rm Ric}_\varphi\geqslant -a^2g$. Then
$$
\aligned
\limsup_{s\rightarrow\infty}\limsup_{l\rightarrow\infty}\int_{1/s}^s\int_{B(o,l)}-\cL^{'}_\varphi& b_{s,l}(x,t)\wrt\mu_\varphi(x)\,t\wrt t\\
& \leqslant \frac{1+\delta}2\,(1+\epsilon)^{p}\left( \nor{f}^p_p+\nor{\omega}^q_q\right)\,.
\endaligned
$$
\end{proposition}
\begin{proof}
Recall that $o\in M$ was fixed at the beginning of this section. Set $r(x)=\rho(x,o)$, where $\rho$ denotes the geodesic distance on $M$.
Thus $B(o,\delta)=\mn{x\in M}{r(x)<\delta}$.  Take a nonincreasing function $\Lambda\in C_c^\infty([0,\infty))$ such that $0\leqslant\Lambda\leqslant 1$, $\Lambda=1$ in $[0,1]$ and $\Lambda=0$ in $[2,\infty)$. For $l>0$ and $x\in M$ define
\begin{equation}
\label{whereimgoing}
F_l(x)=\Lambda\left(\frac{r(x)^2}{l^2}\right)\,.
\end{equation}
Observe that
$(\text{supp}\,F_l)\times[1/s,s]\subset K_{s,l}$. By Proposition~\ref{p: 2}, $-\cL^{'}_\varphi b_{s,l}\geqslant 0$, so that
\begin{eqnarray*}
\aligned
\int_{1/s}^s \int_{B(o,l)}-\cL^{'}_\varphi b_{s,l}(x,t)\wrt\mu_\varphi(x)\,t\wrt t
& \leqslant\int^{s}_{1/s}\int_{M}-\cL^{'}_\varphi b_{s,l}(x,t)F_l(x)\wrt\mu_\varphi(x)\,t\wrt t\\
& =\int^s_{1/s}\int_{M}(\partial^2_{tt}-\cL_\varphi )b_{s,l}(x,t)F_l(x)\wrt\mu_\varphi(x)\,t\wrt t\,.
\endaligned
\end{eqnarray*}
Therefore, to complete the proof it suffices to show that
\begin{equation}
\label{e: 11}
\limsup_{s\rightarrow\infty}\limsup_{l\rightarrow \infty}\int^s_{1/s}\int_{M}\partial^2_{tt}b_{s,l}(x,t)F_l(x)\wrt\mu_\varphi(x)\,t\wrt t\leqslant\frac{1+\delta}2\,(1+\epsilon)^{p} (\|f\|^p_p+\|\omega\|^q_q)
\end{equation}
and
\begin{equation}
\label{e: 12}
\lim_{l\rightarrow \infty}\int^s_{1/s}\int_{M}\cL_\varphi b_{s,l}(x,t)F_l(x)\wrt\mu_\varphi(x)\,t\wrt t=0\,
\end{equation}
for all $s>0$.

\medskip
We first prove \eqref{e: 11}. An integration by parts in the variable $t$ gives
$$
\int^s_{1/s}\partial^2_{tt}b_{s,l}(x,t)t\wrt t=s\partial_tb_{s,l}(x,s)-s^{-1}\partial_tb_{s,l}(x,1/s)+b_{s,l}(x,1/s)-b_{s,l}(x,s).
$$
Theorem~\ref{t: regular bellman} (\ref{anterija'}') and Lemma~\ref{l: 2} imply, for all $(x,t)\in K_{s,l}$,
$$
\aligned
b_{s,l}(x,1/s)-b_{s,l}(x,s)&\leqslant b_{s,l}(x,1/s)\\
& \leqslant \frac{1+\delta}2\,(1+\epsilon)^{p}\left(P^a_{1/s}|f(x)|^p+P^0_{1/s}|\omega(x)|^q\right)\,.
\endaligned
$$
It follows that
$$
\aligned
& \ \hskip -20pt
\int^s_{1/s}\int_{M}\partial^2_{tt}b_{s,l}(x,t)F_l(x)\wrt\mu_\varphi(x)\,t\wrt t\\
&\leqslant \frac{1+\delta}2\,(1+\epsilon)^p\left( \|f\|^p_p+\|\omega\|^q_q\right)
+\nor{s\partial_tb_{s,l}(x,s)-s^{-1}\partial_tb_{s,l}(x,1/s)}_{L^1(\mu_\varphi)},
\endaligned
$$
where in the last inequality we used the fact that for every $r\in[1,\infty]$ the semigroup $P^a_t$ is contractive on $L^r$. Therefore, in order to prove \eqref{e: 11} it is enough to show that
\begin{equation}\label{e: 10}
\lim_{s\rightarrow\infty}
\nor{s\partial_tb_{s,l}(x,s)-s^{-1}\partial_tb_{s,l}(x,1/s)}_{L^1(\mu_\varphi)}=0\,.
\end{equation}
Since the semigroup $P^a_t$ is contractive in $L^r$ for all $r\in [1,\infty]$, there exists $h=h(p)>2$ such that
$$
\nor{(P^a_t|f|)^{p-1}+P^0_t|\omega|+(P^0_t|\omega|)^{q-1}}_{L^h(\mu_\varphi)}\leqslant C(f,\omega,p)
$$
uniformly in $t>0$. Hence, by Lemma \ref{l: 2} and H\"older's inequality, to prove \eqref{e: 10} it suffices to show that
\begin{equation}
\label{masaaki}
\lim_{t\rightarrow 0,\infty}
\Nor{|t\partial_tP^a_tf|+|t\partial_t\vec P^a_t\omega|}_{L^{h'}(\mu_\varphi)}=0\,,
\end{equation}
where $h'$ is the conjugate exponent of $h$. To prove \eqref{masaaki}, simply observe that by the spectral theorem $t\partial_tP^a_tf$ and $t\partial_t\vec{P^a_t}\omega$ converge to $0$ in $L^2$ as $t\rightarrow 0,\infty$, and that $\|t\partial_tP^a_t\|_r+\|t\partial_t\vec P^a_t\|_r$ is uniformly bounded in $t$ for all $r$ in $(1,\infty)$ \cite[Theorem 4.6 (c)]{EN}, because ${\rm Ric}_\varphi\geqslant -a^2g$ implies that the semigroups $P^a_t$ and $\vec P^a_t$ are both analytic on $L^r(\mu_\varphi)$, for all $r$ in $(1,\infty)$ .

\medskip
We now prove \eqref{e: 12}. Since $\|\wrt r\|_\infty\leqslant 1$ and $F_l=0$ for $r\geqslant 2l$, we have that
\begin{equation}\label{eq: 10}
\|\wrt
F_l
\|_\infty\leqslant  \frac{4\|\Lambda'\|_\infty}{l}\,.
\end{equation}
Since $F_l$ is a compactly supported Lipschitz function, by \cite[Lemma 2.5]{HE} it belongs to the first-order Sobolev space over $L^2(\mu_\varphi)$. Therefore, an integration by parts gives
$$
\int^s_{1/s}\int_{M}\cL_\varphi b_{s,l}(x,t)F_l(x)\wrt\mu_\varphi(x)\,t\wrt t=\int^s_{1/s}\int_{M}\langle\wrt b_{s,l}(x,t),\wrt F_l(x)\rangle\wrt\mu_\varphi(x)\,t\wrt t.
$$
We are not able to deduce \eqref{e: 12} directly from \eqref{eq: 10} and the above formula, because under our curvature assumption we do not have any good estimate of $|\wrt b_{s,l}|$. This problem can be fixed by performing another integration by parts.

Since ${\rm Ric}_\varphi\geqslant -a^2g$, by \cite[Theorem 3.1]{wei and wylie} (see also \cite[Theorem 2.4]{pigola} for the unweighted case) we have that
\begin{equation}\label{e: 8}
-\cL_\varphi r(x)\leqslant C\max\left\{\frac{1}{r(x)},r(x)\right\},
\end{equation}
for all $x\in M\setminus({\rm cut}(o)\cup\{o\})$, where ${\rm cut}(o)$ denotes the cut locus of the point $o$. A simple computation based on \eqref{whereimgoing} gives
$$
-\cL_\varphi
F_l(x)=-\frac{2r(x)\cL_\varphi r(x)}{l^2}\Lambda'(r^2/l^2)+\frac{4r^2}{l^4}|\wrt r(x)|^2\Lambda''(r^2(x)/l^2),
$$
for all $x\in M\setminus({\rm cut}(o)\cup\{o\})$. Since $\|\wrt r\|_\infty\leqslant 1$, $\Lambda'\leqslant 0$, and $\Lambda'=0$ on $[0,1]\cup[2,\infty]$, by \eqref{e: 8} there exists $C>0$ such that
\begin{equation}\label{eq: 5}
-\cL_\varphi F_l(x)
\geqslant -C\left(\|\Lambda'\|_\infty +\|\Lambda''\|_{\infty}\right)\chi_{B(o,2l)\setminus B(o,l)}\,
\end{equation}
for all $x\in M\setminus({\rm cut}(o)\cup\{o\})$ and all $l\geqslant 1$. Moreover, a simple modification of the argument used in the proof of \cite[Lemma 2.5]{pigola} shows that \eqref{eq: 5} holds weakly on $M$. In particular, we have
\begin{eqnarray*}
\aligned
\int_{M}-\cL_\varphi b_{s,l}(x,t)F_l(x)\wrt\mu_\varphi(x)\geqslant-C\int_{B(o,2l)\setminus B(o,l)}b_{s,l}(x,t)\wrt\mu_\varphi(x)\,.
\endaligned
\end{eqnarray*}
Hence the first inequality of Lemma~\ref{l: 2} implies that
\begin{eqnarray*}
\aligned
\int_{M}-\cL_\varphi b_{s,l}F_l\wrt\mu_\varphi(x)\geqslant-C\int_{B(o,2l)\setminus B(o,l)}(P^a_t|f|^p+P^0_t|\omega|^q)\wrt\mu_\varphi\,.
\endaligned
\end{eqnarray*}
 Denote the integral on the right-hand side by $\Psi_l(t)$. Since $\lim_{l\rightarrow\infty}\Psi_l=0$ pointwise on $\R_+$ and $0\leqslant \Psi_l(t)\leqslant\|f\|^p_p+\|\omega\|^q_q$, the Lebesgue dominated convergence theorem implies
\begin{equation}\label{e: 13}
\liminf_{l\rightarrow \infty}\int^s_{1/s}\int_{M}-\cL_\varphi b_{s,l}(x,t)F_l(x)\wrt\mu_\varphi(x)\,t\wrt t\geqslant 0.
\end{equation}

\medskip
It remains to prove that
\begin{equation*}
\limsup_{l\rightarrow \infty}\int^s_{1/s}\int_{M}-\cL_\varphi b_{s,l}(x,t)F_l(x)\wrt\mu_\varphi(x)\,t\wrt t\leqslant 0\,.
\end{equation*}
Consider the function $R=(1/2)(1+\delta)(1+\epsilon)^{p}(P^a_t|f|^p+P^0_t|\omega|^q)$. By Lemma~\ref{l: 2}, $b_{s,l}-R\leqslant 0$ on $K_{s,l}$, and an argument similar to the one we used to prove \eqref{e: 13} shows that
$$
\limsup_{l\rightarrow \infty}\int^s_{1/s}\int_{M}-\cL_\varphi (b_{s,l}(x,t)-R(x,t))F_l(x)\wrt\mu_\varphi(x)\,t\wrt t\leqslant 0.
$$
We now prove that
\begin{equation}\label{eq: 6}
\limsup_{l\rightarrow \infty}\int^s_{1/s}\int_{M}\cL_\varphi R(x,t)F_l(x)\wrt\mu_\varphi(x)\,t\wrt t= 0.
\end{equation}

By integrating by parts and using \eqref{eq: 10}, we get
$$
\mod{\int^s_{1/s}\int_{M}
\cL_\varphi R(x,t)F_l(x)\wrt\mu_\varphi(x)\,t\wrt t}\leqslant\frac{4\|\Lambda'\|_\infty}{l}\int^s_{1/s}\int_{M}|\wrt R(x,t)|\wrt\mu_\varphi(x)\,t\wrt t.
$$
By Lemma~\ref{l: commutativita},
$$
\wrt R(x,t)=C(\wrt P^a_t|f|^p+\wrt P^0_t|\omega|^q)=C(\vec{P^a_t}\wrt|f|^p+\vec{P^0_t}\wrt|\omega|^q)\,,
$$
where the right-hand side is in $L^1(M\times[1/s,s], \wrt\mu_\varphi \,t\wrt t)$ because $f$ and $\omega$ are regular, compactly supported and $|\vec{P^0_t}\eta|\leqslant e^{ta^2}P^0_t|\eta|$ for all $\eta\in C^\infty(T^*M)$. This implies \eqref{eq: 6} and concludes the proof of the proposition.
\end{proof}

\begin{proof}
[Proof of Theorem \ref{t: 3}.]
Suppose that $p\geqslant 2$. By combining Propositions \ref{p: 2} and \ref{p: 3}, using the Fatou lemma and passing to the limit as $\epsilon\rightarrow0$, we get
$$
2\delta\int^\infty_0\int_M|\overline{\nabla}P^a_tf||\overline{\nabla}\vec{P^a_t}\omega|\wrt\mu_\varphi\, t\wrt t\leq
\frac{1+\delta}2\,\big(\|f\|^p_p+\|\omega\|^q_q\big)\,.
$$
Now apply the above inequality to $\lambda f$ and $\lambda^{-1}\omega$ instead of $f,\omega$, respectively, and minimize in $\lambda>0$.
The result is
$$
\int^\infty_0\int_M|\overline{\nabla}P^a_tf||\overline{\nabla}\vec{P^a_t}\omega|\wrt\mu_\varphi\, t\wrt t\leq
C_q(p-1)\|f\|_p\|\omega\|_q\,,
$$
where
$$
C_q=\frac{8+q(q-1)}{4}\,(q-1)^{1/q-1}\,.
$$
The substitution $s=q-1$ returns
$$
\sup_{q\in (1,2]}C_q=\frac14\,\sup_{s\in (0,1]}(s^2+s+8)s^{-\frac{s}{s+1}}< 2^\cdot8<3\,.
$$

When $1<p<2$, interchange $P^a_tf$ and $\vec{P^a_t}\omega$ in the definition of $b_{\kappa_{s,l}}$ and proceed as before.
\end{proof}

\section*{Acknowledgements}
A part of this work was conducted at Centro Di Ricerca Matematica Ennio De Giorgi, Pisa, during the intensive research period in Euclidean Harmonic Analysis, Nilpotent Lie Groups and PDEs 2010, and during the first author's visit to the University of Ljubljana. The authors thank the faculty and staff of those institutions for their warm hospitality.

\end{document}